\documentclass{amsart}
\usepackage[mathscr]{eucal}
\usepackage{amssymb}
\usepackage{latexsym}
\usepackage{graphicx}
\usepackage[dvipdfm, usenames]{color}

\theoremstyle{plain}
\newtheorem{theorem}{Theorem}

\newtheorem{lemma}{Lemma}
\newtheorem{proposition}{Proposition}

\theoremstyle{definition}
\newtheorem{definition}{Definition}
\newtheorem{remark}{Remark}

\newcommand{\C}{\mathbb{C}}

\title
[Strongly projectively flat maps]{Equivariant strongly projectively flat maps of compact homogeneous K\"ahler manifolds}
\author{Isami Koga}

\address{Graduate school of Mathematics,\\
Kyushu university,\\
744 Motooka, Nishi-ku, Fukuoka, 819-0395, Japan.}
\email{i-koga@math.kyushu-u.ac.jp}

\begin{document}
\maketitle
\begin{abstract}
	In \cite{Koga1}, the author define projectively flat maps of a compact K\"{a}hler manifold into complex 
	Grassmannian manifold.
	In this article, by focusing on the essence of the result in \cite{Koga1} we define {\it strongly projectively flat maps} and study such maps.
	Finally we prove a rigidity of equivariant strongly projectively flat maps of simply connected 
	homogeneous K\"{a}hler manifold.
\end{abstract}

\section{Introduction}
	Holomorphic maps into the complex projective space have been studied for a long time.
	In \cite{Calabi1} E. Calabi have proved that holomorphic isometric immersions of K\"ahler manifolds into the complex 
	projective space are rigid and equivariant with respect to the group of automorphisms of the domain.
	In \cite{Takeuchi} M. Takeuchi has constructed all holomorphic isometric immersions of homogeneous K\"ahler 
	manifolds into the complex projective space and has classified the holomorphic isometric immersions of Hermitian 
	symmetric spaces.

	In this article, we study holomorphic maps of a compact homogeneous K\"ahler manifold into the complex Grassmannian 
	manifold.

	Let $\C^n$ be an $n$-dimensional complex vector space with a Hermitian inner product $(\cdot,\cdot)_n$ and 
	$Gr_p(\C^n)$ the complex Grassmannian manifold of complex $p$-planes in $\C^n$ with the Hermitian metric of 
	Fubini-Study 
	type induced by $(\cdot,\cdot)_n$.
	Let $M$ be a compact K\"ahler manifold and $f:M\longrightarrow Gr_p(\C^n)$ be a holomorphic map.
	Then we have a holomorphic vector bundle $f^*Q\rightarrow M$ over $M$ which is the pull-back bundle of the universal 
	quotient bundle $Q\rightarrow Gr_p(\C^n)$ over $M$ by $f$.
	
	\begin{definition}[cf.\cite{Koga1}]
	\label{def1}
	\begin{enumerate}
		\item[]
		\item[(1)]A holomorphic map $f:M\longrightarrow Gr_p(\C^n)$ is called {\it projectively flat} if the pull-back 
		bundle $f^*Q\rightarrow M$ is projectively flat.	
		\item[(2)]A holomorphic map $f:M\longrightarrow Gr_p(\C^n)$ is called {\it strongly projectively flat} if there 
		exists a holomorphic Hermitian line bundle $L\rightarrow M$ such that $f^*Q\rightarrow M$ is isomorphic to $\tilde{L}\rightarrow M$ with holomorphic structures and fiber metrics, where $\tilde{L}\rightarrow M$ is orthogonal direct sum of $q$-copies of $L\rightarrow M$.
	\end{enumerate}
	\end{definition}

	Strongly projectively flat condition is a kind of simple extension of a map into the complex projective space with Fubini-Study metric. (For a detail, see the section 3).
	
	Projectively flat condition is defined by the author in \cite{K1}.
	A strongly projectively flat map is projectively flat since the orthogonal direct sum of $q$-copies of a holomorphic 
	line bundle is projectively flat.
	In general, the inverse of this asseertion is not true.
	However, we can show the following assertion.
	
	\begin{proposition}
	\label{prop1}
		Assume that a holomorphic map $f:M\rightarrow Gr_p(\C^n)$ is an isometric.
		Then $f$ is strongly projectively flat if and only if $f$ is projectively flat.
	\end{proposition}
	\begin{proof}
		For a detail, see \cite{Koga1}.
		Let $M$ be a compact K\"ahler manifold and $f:M\longrightarrow Gr_p(\C^n)$ be a holomorphic isometric 
		projectively flat immersion.
		It follows from the result in \cite{Koga1} that the curvature $R^{f^*Q}$ of pull-back bundle $f^*Q\rightarrow M$ 
		is expressed as
			\begin{equation*}
			R^{f^*Q} = -\cfrac{\sqrt{-1}}{q}\omega_M{\rm Id}_{Q_x},
			\end{equation*}
		where $q$ is the rank of $Q\rightarrow M$, $\omega_M$ is the K\"ahler form on $M$ and ${\rm Id}_{Q_x}$ is the 
		identity map of the fiber of $f^*Q\rightarrow M$ at $x\in M$.
		Since $\omega_M$ is parallel, it follows from the holonomy theorem that there exists a Hermitian line bundle $L
		\rightarrow M$ such that $f^*Q\rightarrow M$ is isomorphic to the orthogonal direct sum of $q$-copies of $L
		\rightarrow M$ as a Hermitian vector bundle.
		
		Therefore $f:M\longrightarrow Gr_p(\C^n)$ is strongly projectively flat.
	\end{proof}

	In the present paper, our goal is to show the following theorem.
	
	\begin{theorem}
		Let $M=G/K$ be a compact simply connected K\"ahler manifold such that $G$ is the isometry group of $M$ and $K$ an isotropy subgrouop of $G$.
		Let $f:M\longrightarrow Gr_p(\C^n)$ be a full holomorphic strongly projectively flat map into the complex 
		Grassmannian manifold.
		If $f$ is $G$-equivariant, then there exists an $N$-dimensional complex vector space $W$ and a holomorphic map 
		$f_0:M\longrightarrow Gr_{N-1}(W)$ such that $\C^n$ is regarded as the orthogonal direct sum of $q$-copies of $W
		$ and $f$ is congruent to the following composed map:
		\begin{align}
			\label{01}
			\begin{split}
				f:& M\longrightarrow Gr_{N-1}(W)\times\cdots\times Gr_{N-1}(W)\longrightarrow Gr_{q(N-1)}(\C^n),
				\\
				&x\longmapsto (f_0(x),\cdots,f_0(x))\longmapsto f_0(x)\oplus\cdots\oplus f_0(x).
			\end{split}
		\end{align}
	\end{theorem}
	
	$G$-equivariance of a holomorhic map $f:M\longrightarrow Gr_p(\C^n)$ means that there exists a Lie group 
	homomorphism $\rho:G\longrightarrow SU(n)$ which satisfies the following equaiton:
		\begin{equation}
		\label{15}
		f(gx) = \rho(g)f(x),\qquad {\rm for}\ x\in M, g\in G.
		\end{equation}
	
	The author would like to thank Professor Yasuyuki Nagatomo for his many advices and encouragement. 
	
\section{Preliminaries}
	For a detail of the argument of this section, see \cite{N1}.
	Let $\C^n$ be an $n$-dimensional complex vector space with a Hermitian inner product $(\cdot,\cdot)_n$ and 
	$Gr_p(\C^n)$ be the complex Grassmannian manifold of complex $p$-planes in $\C^n$.
	We denote by $S\rightarrow Gr$ the tautological bundle and by $\underline{\C^n}:=Gr_p(\C^n)\times \C^n\rightarrow Gr
	$ the trivial bundle over $Gr_p(\C^n)$.
	They are holomorphic vector bundles.
	The trivial bundle $\underline{\C^n}\rightarrow Gr$ has a Hermitian fiber metric induced by $(\cdot,\cdot)_n$, which  
	is denoted by the same notation.
	Since $S\rightarrow Gr$ is a subbundle of $\underline{\C^n}\rightarrow Gr$, The bundle $S\rightarrow Gr$ has a 
	Hermitian fiber metric $h_S$ induced by $(\cdot,\cdot)_n$ and we obtain a holomorphic vector bundle $Q\rightarrow 
	Gr$ satisfying the following short exact sequence:
		\begin{equation}
		\label{1}
		0\longrightarrow S\longrightarrow \underline{\C^n}\longrightarrow Q\longrightarrow 0.
		\end{equation}
	This is called the {\it universal quotient bundle} over $Gr_p(\C^n)$.
	When we denote by $S^{\perp}\rightarrow Gr$ the orthogonal complement bundle of $S\rightarrow Gr$ in $
	\underline{\C^n}\rightarrow Gr$, $Q\rightarrow Gr$ is isomorphic to $S^{\perp}\rightarrow Gr$ as a 
	$C^{\infty}$-complex 
	vector bundle.
	Thus $Q\rightarrow Gr$ has a Hermitian fiber metric $h_Q$ induced by the Hermitian fiber metric of $S^{\perp}
	\rightarrow Gr$.

	These vector bundles are all homogeneous vector bundles.
	We set $\tilde{G}:=SU(n)$ and $\tilde{K}:=S(U(p)\times U(q))$.
	Then $Gr_p(\C^n)\cong\tilde{G}/\tilde{K}$.
	Let $\C^p$ be a $p$-dimensional complex subspace of $\C^n$ such that $\C^p$ is an irreducible representation space of $\tilde{K}$.
	We denote by $\C^q$ the orthogonal complement space of $\C^p$ in $\C^n$, which is also an irreducble representation space of $\tilde{K}$.
	Then $S\rightarrow Gr$ and $S^{\perp}\rightarrow Gr$ are expressed as the following:
		\begin{equation*}
		S=\tilde{G}\times_{\tilde{K}}\C^p, \qquad S^{\perp} = \tilde{G}\times_{\tilde{K}}\C^q.
		\end{equation*}
	For the exact sequence (\ref{1}), the inclusion $S\longrightarrow \underline{\C^n}$ is expressed as
		\begin{equation*}
		S=\tilde{G}\times_{\tilde{K}}\C^p\ni [g,v]\longmapsto ([g],gv)\in \tilde{G}/\tilde{K}\times \C^n = 
		\underline{\C^n},
		\end{equation*}
	for $g\in \tilde{G}$ and $v\in \C^p$.
	Similarly $Q\rightarrow Gr$ is regarded as a subbundle of $\underline{\C^n}\rightarrow M$:
		\begin{equation*}
		Q\cong S^{\perp}\ni [g,v]\longmapsto ([g],gv)\in \underline{\C^n},
		\end{equation*}
	for $g\in\tilde{G}$ and $v\in \C^q$.
	When we regard $Q\rightarrow Gr$ a subbundle of $\underline{\C^n}\rightarrow Gr$ as above, the action of $G$ to $Q
	\rightarrow Gr$ is expressed as the following:
		\begin{equation}
		\label{2}
		g\cdot([\tilde{g}],\tilde{g}v) = (g\cdot[\tilde{g}],g\tilde{g}v),\qquad {\rm for}\ g,\tilde{g}\in \tilde{G}, v
		\in
		\C^q.
		\end{equation}
	
	Since the holomorphic tangent bundle $T_{1,0}Gr\rightarrow Gr$ over $Gr_p(\C^n)$ is identified with $S^*\otimes Q
	\rightarrow Gr$, where $S^*\rightarrow Gr$ is the dual bundle of $S\rightarrow Gr$, complex manifold $Gr_p(\C^n)$
	has a homogeneous Hermitian metric $h_{Gr}:=h_{S^*}\otimes h_Q$.
	This is called the Hermitian metric of Fubini-Study type of $Gr_p(\C^n)$ induced by $(\cdot,\cdot)_n$.
	
\begin{remark}
	When we consider the case that $p=n-1$\footnote{In this paper, the complex projective space means the complex 
	Grassmannian manifold $Gr_{n-1}(\C^n)$, not $Gr_1(\C^n)$.}, $(Gr_{n-1}(\C^n), h_{Gr})$ is the complex projective 
	space 
	with Fubini-Study metric of constant holomophic sectional curvature 2. (See \cite{Koga1}.)
\end{remark}

	Let $M$ be a compact complex manifold, $V\rightarrow M$ a holomorphic vector bundle with a Hermitian fiber 
	metric $h_V$ and $W$ the space of holomorphic sections of $V\rightarrow M$.
	We denote by $(\cdot,\cdot)_W$ the $L_2$-Hermitian inner product of $W$.
	Let $\tilde{W}$ be a subspace of $W$ and we denote by $ev$ a bundle homomorphism:
		\begin{equation*}
		ev:\underline{\tilde{W}}:=M\times \tilde{W} \longrightarrow V : (x,t)\longmapsto t(x).
		\end{equation*}
	This is called an {\it evaluation map}.
	We assume that the evaluation map $ev$ is surjective. In this case, $V\rightarrow M$ is called {\it globally 
	generated by $\tilde{W}$}.
	For each $x\in M$ we set the linear map $ev_x:\tilde{W}\longrightarrow V_x:t\longmapsto t(x)$.
	Since $V\rightarrow M$ is globally generated by $\tilde{W}$, $ev_x$ is surjective for each $x\in M$.
	Thus the dimension of the kernel ${\rm Ker}ev_x$ of $ev_x$ are independent of $x$, which is denoted by $p$.
	Therefore we obtain a holomorphic map
		\begin{equation*}
		f:M\longrightarrow Gr_p(\tilde{W}) : x\longmapsto {\rm Ker}ev_x.
		\end{equation*}
	This is called a {\it induced map} by $(L\to M, \tilde{W})$.
	When $\tilde{W}=W$, the induced map is called {\it standard map} by $L\to M$.
	
	On the other hand, let $f:M\longrightarrow Gr_p(\C^n)$ be a holomorphic map and $f^*Q\rightarrow M$ the pull-back 
	vector bundle of $Q\rightarrow Gr$ by $f$ with pull-back metric $h_Q$ and connection $\nabla^Q$.
	We assume that $f^*Q\rightarrow M$ is isomorphic to $V\rightarrow M$ as a holomorphic Hermitian vector bundle.
	For \cite{N1} there exists a semi-positive Hermitian endomorphism $T$ of $W$ such that $f$ is expressed as the 
	following map:
		\begin{equation}
		\label{3}
		f:M\longrightarrow Gr_{p}(W):x\longmapsto T^{-1}\left( f_0(x)\cap ({\rm Ker}\ T)^{\perp} \right),
		\end{equation}
	where $T^{-1}$ is the inverse of $T:{\rm Ker}T^{\perp}\longrightarrow {\rm Ker}T^{\perp}$.
	The semi-positive Hermitian endomorphism $T:W\longrightarrow W$ is obtaind by the following construction:
	by Borel-Weil Theory the complex vector space $\C^n$ is regarded as the space of holomorphic sections of $Q
	\rightarrow Gr_p(\C^n)$.
	We have a linear map $\iota:\C^n\longrightarrow W$ by restricting sections to $M$.
	
	\begin{definition}[\cite{N1}]
	\label{def2}
		A holomorphic map $f:M\longrightarrow Gr_p(\C^n)$ is called {\it full} if $\iota:\C^n\longrightarrow W$ is 
		injective. 
	\end{definition}	
	We assume that $f:M\longrightarrow Gr_p(\C^n)$ is full. 
	Then $\C^n$ can be considered as a subspace of $W$ by $\iota$.
	Let $ev_{\C}:\underline{\C^n}\longrightarrow V$ and $ev:\underline{W}\longrightarrow V$ be evaluation maps.
	Then for any $x\in M$, ${\rm Ker}ev_{\C_x}={\rm Ker}ev_x\cap \C^n$.
	Therefore $f:M\longrightarrow (Gr_p(\C^n),(\cdot,\cdot)_n)$ is expressed as 
		\begin{equation*}
		 f(x)= {\rm Ker}ev_x\cap \C^n. 
		\end{equation*}
	The Hermitian inner product $(\cdot,\cdot)_n$ is not always coincide with $(\cdot,\cdot)_W$.
	Let $\underline{T}$ be the positive Hermitian endomorphism of $\C^n$ which satisfies that
		\begin{equation}
		\label{4}
		(\underline{T}u,\underline{T}v)_n = (u,v)_W
		\end{equation}
	for any $u,v\in \C^n$.
	We have an isometry
		$$\underline{T}^{-1}:(Gr_p(\C^n),(\cdot,\cdot)_n)\longrightarrow (Gr_p(\C^n),(\cdot,\cdot)_W): U\longmapsto 
		\underline{T}^{-1}U.$$
	Let $\pi:W\longrightarrow \C^n$ be the orthogonal projection onto $\C^n$ with respect to $(\cdot,\cdot)_W$.
	We denote by $T:=\underline{T}\circ\pi$ an endomophism of $W$, which is semi-positive Hermitian.
	Consequently $f:M\longrightarrow Gr_p(\C^n)$ is expressed as
		\begin{equation}
		\label{5}
		f:M\longrightarrow (Gr_p(W),(\cdot,\cdot)_W):x\longmapsto T^{-1}(f_0(x)\cap{\rm Ker}T^{\perp}).
		\end{equation}
	This means that a holomorphic map which has the pull-back bundle $f^*Q\rightarrow M$ isomorphic to $V\rightarrow M$ 
	is expressed a deformation of the standard map induced by $V\rightarrow M$.
	
	Let $M=G/K$ be a compact homogeneous K\"{a}hler manifold, $V_0$ a irreducible $K$-representation space and $V=G
	\times_KV_0\rightarrow M$ a holomorphic homogeneous vector bundle.
	We denote by $W$ the space of holomoprhic sections of $V$ with $L_2$-Hermitian inner product $(\cdot,\cdot)_W$.
	It follows from Bott-Borel-Weil theory that $W$ is an irreducible $G$-representation space.
	For $g\in G$ and $t\in W$, the action of $G$ to $W$ is expressed as
		\begin{equation*}
		(g\cdot t)([g_0]) := g(t(g^{-1}[g_0])),\qquad g_0\in G.
		\end{equation*}
	
	We assume that the evaluation map $ev:M\times W\longrightarrow V$ is surjective.

	\begin{proposition}[\cite{N1}]
	\label{prop2}
		$V_0$ is considered as a subspace of $W$.
	\end{proposition}
	
	We set $\pi_0:= ev_{[e]}:W\longrightarrow V_0$.
	For $g\in G$ and $t\in W$, we can calculate
		\begin{align}
		\label{13}
		\begin{split}
		ev([g],t)&=t([g])=g\cdot g^{-1}t(g\cdot[e]) \\
		&=g\cdot ev([e], g^{-1}t)=g[e,\pi_0(g^{-1}t)]=[g,\pi_0(g^{-1}t)].
		\end{split}
		\end{align}
	The adjoint map $ev^*$ of $ev$ is expressed as $ev^*([g,v]) = ([g],gv)$.

	Let $U_0:={\rm Ker}ev_{[e]}$.
	Then $U_0$ is also $K$-representation space and $W$ is orthogonal direct sum of $U_0$ and $V_0$.
	It follows from (\ref{13}) that ${\rm Ker}ev_{[g]}= gU_0$.
	Therefore the standard map $f_0:M\longrightarrow Gr(W)$ induced by $V\rightarrow M$ is that
		\begin{equation}
		\label{6}
		f_0([g]) = gU_0.
		\end{equation}	
	Thus the standard map induced by a holomorphic homogeneous vecor bundle is $G$-equivariant.
	
	Let $f:M\longrightarrow Gr_p(\C^n)$ be a full holomorphic map such that $f^*Q\rightarrow M$ is holomorphic 
	isomorphic to $V\rightarrow M$.
	Then there exists a semi-positive Hermitian endomorphism $T:W\longrightarrow W$ such that the map $f$ is expressed as the following:
		\begin{equation*}
		f:M\longrightarrow (Gr_p(\C^n),(\cdot,\cdot)_W):
		[g] \longmapsto T^{-1}( gU_0\cap ({\rm Ker}T)^{\perp}).
		\end{equation*}
	Pulling the sequence (\ref{1}) back, we have a short exact sequence:
		\begin{equation}
		\label{7}
		0\longrightarrow f^*S\longrightarrow \underline{\C^n}:= M\times \C^n\longrightarrow f^*Q\longrightarrow 0.
		\end{equation}
	The vector bundle $f^*Q\rightarrow M$ is regarded as a subbundle of $\underline{\C^n}\rightarrow M$, which is 
	orthogonal complement bundle of $f^*S\rightarrow M$:
		\begin{equation*}
		f^*Q = \{ ([g], v)\in \underline{\C^n} | f([g]) \perp v \}.
		\end{equation*}
	It follows from (\ref{3}) that we have
		\begin{equation*}
		0=(v,T^{-1}(gU_0\cap \C^n))_W=(g^{-1}T^{-1}v,U_0\cap \C^n)_W.
		\end{equation*}
	Then we have $g^{-1}T^{-1}v\in V_0 \iff v\in TgV_0$.
	Therefore the isomorphism $\phi:V\longrightarrow f^*Q\subset \underline{\C^n}$ is expressed as the following:
		\begin{equation}
		\label{8}
		[g,v]\longmapsto ([g], Tgv),\qquad {\rm for}\ g\in G,\ v\in V_0.
		\end{equation}

\section{Main Theorem and its Proof}

	Let $M$ be a compact simply connected homogeneous K\"ahler manifold, $G$ the isometry group of $M$ and $K$ an isotropy subgroup of $G$.
	Let $f:M\longrightarrow Gr_p(\C^n)$ be a full holomorphic strongly projectively flat map.
	By definition of strongly projectively flatness there exists a Hermitian line bundle $L\rightarrow M$ such that 
	$f^*Q\rightarrow M$ is isomorphic to $\tilde{L}\rightarrow M$ as a Hermitian vector bundle, where $\tilde{L}
	\rightarrow M$ is orthogonal direct sum of $q$-copies of $L\rightarrow M$.
	Since $M$ is a compact simply connected homogeneous K\"ahler, $L\rightarrow M$ is homogeneous.
	We set $L_0$ the $1$-dimensional $K$-representation space such that $L = G\times_KL_0$.
	Then  we have
		\begin{equation*}
		\tilde{L}=L\oplus\cdots\oplus L= G\times_K(L_0\oplus\cdots\oplus L_0)=G\times_K\tilde{L}_0,
		\end{equation*}
	where $\tilde{L}_0$ is $q$-orthogonal direct sum of $L_0$.
	We denote by $W$ and $\tilde{W}$ the spaces of holomorphic sections of $L\rightarrow M$ and $\tilde{L}\rightarrow M$ 
	respectively and we set $N$ the dimension of $W$.
	By definition of $\tilde{L}\rightarrow M$, $\tilde{W}$ is regarded as $q$-orthogonal direct sum of $W$.
	Let $\pi_j:\tilde{W}\longrightarrow W$ be the orthogonal projection onto the $j$-th component of $\tilde{W}$.
	It follows from Proposition \ref{prop2} that $L_0$ is a subspace of $W$ and $\tilde{L}_0$ is a subspace of $
	\tilde{W}$ as a $K$-representation space.
	When we restrict $\pi_j$ to $\tilde{L}_0$, $\pi_j|_{\tilde{L}_0}$ is orthogonal projection of $\tilde{L}_0$ onto the 
	$j$-th component of $\tilde{L}_0$.
	We denote by
		\begin{equation*}
		ev:M\times W\longrightarrow L,\qquad
		\tilde{ev}:M\times \tilde{W}\longrightarrow \tilde{L}
		\end{equation*}
	the evaluation maps respectively and by $f_0:M\longrightarrow Gr_{N-1}(W)$ and $\tilde{f}_0:M\longrightarrow 
	Gr_{q(N-1)}(\tilde{W})$ the standard maps induced by $L\rightarrow M$ and $\tilde{L}\rightarrow M$ respectively.
	Since $\tilde{W}$ is orthogonal direct sum of $q$-copies of $W$, we have
		\begin{align}
		\label{14}
		\begin{split}
		\tilde{ev}([g],t)&=\tilde{ev}([g],t_1\oplus\cdots\oplus t_q) \\
		&=ev([g],t_1)\oplus\cdots\oplus ev([g],t_q)\in L\oplus
		\cdots\oplus L,
		\end{split}
		\end{align}
	where $t=t_1\oplus\cdots\oplus t_q$ is orthogonal decomposition with respect to $\tilde{W}=W\oplus\cdots\oplus W$.
	We set $U_0:={\rm Ker}ev_{[e]}$.
	Then it follows from (\ref{6}) that $f_0([g])=gU_0$.
	It follows from (\ref{14}) that the map $\tilde{f}_0$ is expressed as
		\begin{align}
		\label{9}
		\begin{split}
		\tilde{f}_0:& M\longrightarrow Gr_{N-1}(W)\times\cdots\times Gr_{N-1}(W)\longrightarrow Gr_{q(N-1)}(\tilde{W}),
		\\
		&[g]\longmapsto (gU_0,\cdots,gU_0)\longmapsto gU_0\oplus\cdots\oplus gU_0=g\cdot(U_0\oplus\cdots\oplus U_0).
		\end{split}
		\end{align}
	
	Since $f:M\longrightarrow Gr_p(\C^n)$ is full, $\C^n$ is a subspace of $\tilde{W}$.
	It follows from the previous section that there exists a semi-positive Hermitian endomorphism $T$ of $\tilde{W}$ and 
	a bundle isomorphism $\phi:\tilde{L}\longrightarrow f^*Q$ such that maps $f:M\longrightarrow Gr_p(\C^n)$ and $\phi:
	\tilde{L}\longrightarrow f^*Q$ is expressed as the 
	following:
		\begin{align}
		\label{10}
		& f([g]) = T^{-1}\left( \tilde{f}_0([g])\cap ({\rm Ker}T)^{\perp}\right), \\
		\label{11}
		& \phi([g,v]) = ([g], Tgv),
		\end{align}
	for $g\in G$ and $v\in \tilde{L}_0$.
	
	By using the above notations, we rewrite Main theorem as followings:
	
	\begin{theorem}
		Let $M$ be a compact simply connected homogeneous K\"ahler manifold and $G$ the isometry group of $K$.
		Let $f:M\longrightarrow Gr_p(\C^n)$ be a full holomorphic strongly projectively flat map into the complex 
		Grassmannian manifold.
		Then $f$ is $G$-equivariant if and onlyl if $f$ is the standard map.
	\end{theorem}
	
	In order to prove this theorem, it is sufficient to show that the Hermitian endomorphism $T:\tilde{W}\longrightarrow 
	\tilde{W}$ is the identity map of $\tilde{W}$.

	From now on, we assume that $f:M\longrightarrow Gr_p(\C^n)$ is $G$-equivariant. 
	Then there exists a Lie group homomorphism $\rho:G\longrightarrow SU(n)$ which satisfy the following equation:
		\begin{equation}
		\label{12}
		f(g[\tilde{g}]) = \rho(g)f([\tilde{g}]),\qquad g,\tilde{g}\in G.
		\end{equation}
	By definition $\C^n$ is $G$-representation space and a vector subspace of $\tilde{W}$.
	
	\begin{lemma}
	\label{lem1}
		$f^*Q\rightarrow M$ is homogeneous.
	\end{lemma}
	\begin{proof}
		The definition of the pull-back bundle $f^*Q\to M$ is that
			\begin{equation*}
			f^*Q= \{ ([g],v) \in M\times Q | f([g]) = \pi(v) \},
			\end{equation*}
		where $\pi:Q\to Gr_p(\C^n)$ is the natural projection.
		For any $([\tilde{g}],v) \in f^*Q$ and $g\in G$, we have an action of $G$ to $f^*Q\rightarrow M$ by
			\begin{equation*}
			g\cdot([\tilde{g}],v) = (g[\tilde{g}],\rho(g)v).
			\end{equation*}
		Since $G$ acts to $M$ transitively, $f^*Q\rightarrow M$ is homogeneous.
	\end{proof}
	
	Since $f^*Q\rightarrow M$ is homogeneous, the space of holomorphic sections of $f^*Q\rightarrow M$ is $G$-
	representation space.
	Let $t$ be a holomorphic section of $f^*Q\rightarrow M$.
	For $g\in G$ and $x\in M$, we have
		\begin{equation*}
		(g\cdot t)(x) = g\left( t(g^{-1}x) \right).
		\end{equation*}
	For $t\in \C^n$, we obtain a holomorphic section of $f^*Q\rightarrow M$ which is expressed as
		\begin{equation*}
		t(x) = (x,t(f(x))),\qquad {\rm for}\ x\in M.
		\end{equation*}
	Thus for $g,\tilde{g}\in G$ we obtain
		\begin{align*}
		(g\cdot t)(x) &= g(t(g^{-1}x)) = g\left(g^{-1}x, t(f(g^{-1}x))\right),\\
		&= \left(x,\rho(g)t(\rho(g^{-1})f(x))\right) = \left(x,(\rho(g)t)(f(x))\right),\\
		&= (\rho(g)t)(x).
		\end{align*}
	Therefore $\C^n$ is a $G$-representation subspace of the space of holomorphic sections of $f^*Q\rightarrow M$.

	\begin{lemma}
	\label{lem2}
		The holomorphic isomorphism $\phi:\tilde{L}\longrightarrow f^*Q$ is $G$-equivariant.
	\end{lemma}
	\begin{proof}
		At first we show that $f^*Q\to M$ is isomorphic to $\tilde{L}$ as a homogeneous vector bundle.
		Since $\phi$ preserves Hermitian connections, bundles $\tilde{L}\to M$ and $f^*Q\to M$ have same holonomy groups and $\phi$ is holonomy equivariant.
		Since the action of $K$ to $f^*Q\to M$ and $L\to M$ at $[e]$, where $e$ is the unit element in $G$, is expressed as a action of the holonomy group, $\phi$ is $K$-equivariant.
		Thus $f^*Q_{[e]}$ is isomorphic to $L_0\oplus\cdots\oplus L_0$ as a $K$-representation space.
		Therefore $f^*Q\to M$ is isomorphic to $\tilde{L}\to M$ as a homogeneous vector bundle.
		
		Finally we show that a holomorphic isomorphism $\phi:\tilde{L}\longrightarrow \tilde{L}$ is $G$-equivariant.
		We denote by $\tilde{L}\cong L_1\oplus\cdots\oplus L_q$ and $L_j=G\times_KL_{(j)}$, where $L_j\rightarrow M$ is the $j$-th component and $L_{(j)}$ is isomorphic to $L_0$ as a $K$-representation space for $j=1,\cdots,q$.
		Then we have
			\begin{equation*}
			\tilde{L} = G\times_K(L_{(1)}\oplus\cdots\oplus L_{(q)}).
			\end{equation*}
		
		Let $\phi_j:L_j\longrightarrow \tilde{L}$ be the restriction of $\phi:\tilde{L}\longrightarrow \tilde{L}$ to $L_j\rightarrow M$.
		Then $\phi_j$ is expressed as the following:
			\begin{equation*}
			\phi_j([g,v]) = [g,\varphi_1(g)(v)\oplus\cdots\oplus\varphi_q(g)(v)],\qquad {\rm for}\ g\in G, v\in L_{(j)},
			\end{equation*}
		where $\varphi_i(g):L_{(j)}\longrightarrow L_{(i)}$ is a linear map for $i=1,\cdots,q$.
		Since $L_{(i)}$ and $L_{(j)}$ are isomorphic $1$-dimensional $K$-representation spaces, there exists a complex number $\alpha_i(g)$ such that $\varphi_i(g)(v)=\alpha_i(g)v$.
		Thus we have
			\begin{equation}
			\label{16}
			\phi_j([g,v]) = [g,\alpha_1(g)v\oplus\cdots\oplus\alpha_q(g)v],\qquad {\rm for}\ g\in G, v\in L_{(j)}.
			\end{equation}
		Since $\phi_j$ is a bundle homomorphism, we obtain
			\begin{align*}
				&\phi_j([gk,v])=[gk,\alpha_1(gk)v\oplus\cdots\oplus\alpha_q(gk)v]=[g,\alpha_1(gk)kv\oplus\cdots\oplus\alpha_q(gk)kv], \\
				&\phi_j([g,kv])=[g,\alpha_1(g)kv\oplus\cdots\oplus\alpha_q(g)kv],
			\end{align*}
		for $g\in G$, $k\in K$ and $v\in L_{(j)}$.
		It follows that $\alpha_i(gk)= \alpha_i(g)$ for $i=1\cdots,q$, $g\in G$ and $k\in K$.
		Therefore $\alpha_i$ is a complex valued function on $G/K$.
		Since $\phi_j$ is holomorphic, so is $\alpha_i$ for $i=1,\cdots q$, which implies that $\alpha_i$ is a constant function for each $i$ because $G/K$ is compact.
		We regard $\alpha_i$ as a complex number.
		Then we have
			\begin{equation*}
			\phi_j([g,v]) = [g,\alpha_1v\oplus\cdots\oplus\alpha_qv],\qquad {\rm for}\ g\in G,v\in L_{(j)}.
			\end{equation*}
		This is $G$-equivariant for each $j=1,\cdots,q$.
		Consequently $\phi:\tilde{L}\longrightarrow \tilde{L}$ is $G$-equivariant.
	\end{proof}
	
	It follows from Lemma \ref{lem1} and Lemma \ref{lem2} that $\C^n$ is a $G$-representation subspace of $\C^n$.


	\begin{lemma}
	\label{lem3}
		The semi-positive Hermitian endomorphism $T:\tilde{W}\longrightarrow \tilde{W}$ is $G$-equivariant.
	\end{lemma}
	\begin{proof}
		Since $\phi:\tilde{L}\longrightarrow f^*Q$ is $G$-equivariant, it follows from (\ref{2}) and (\ref{11}) that we 
		can calculate
			\begin{align*}
			& \phi(g_1\cdot[g_2,v])=\phi([g_1g_2,v])=([g_1g_2],Tg_1g_2v), \\
			& \phi(g_1\cdot[g_2,v])=g_1\cdot\phi([g_2,v])=g_1\cdot([g_2],Tg_2v)=([g_1g_2],g_1Tg_2v).
			\end{align*}
		Therefore we have
			\begin{equation*}
			Tgv = gTv,\qquad {\rm for}\ g\in G, v\in \tilde{L}_0.
			\end{equation*}
		We denote by $GL_0$ an subspace of $W$ spanned by $gv$ for any $g\in G$ and $v\in L_0$ and similarly we denote 
		by $G\tilde{L}_0$.
		Then $G\tilde{L}_0$ is regarded as $q$-orthogonal direct sum of $GL_0$.
		$GL_0$ is a $G$-representation subspace of $W$. 
		Since $W$ is irreducible and $GL_0$ is not empty, we obtain $W= GL_0$.
		Consequently we obtain $\tilde{W} = G\tilde{L}_0$.
		It follows that for any $w\in \tilde{W}$ there exists $\alpha_i \in \C$, $g_i\in G$ and $v_i\in\tilde{L}_0$ such that $w=\sum \alpha_ig_iv_i$, where the right hand side of this equation is a finite sum.
		For any $g\in G$, we have
			\begin{equation*}
			Tgw=Tg\sum \alpha_ig_iv_i = \sum Tg\alpha_ig_iv_i = \sum gT\alpha_ig_iv_i = gT\sum \alpha_ig_iv_i = gTw.
			\end{equation*}
		Therefore $T$ is $G$-equivariant.
	\end{proof}
	
	Since $T:\tilde{W}\longrightarrow \tilde{W}$ is $G$-equivariant, $T$ is also $K$-eqivariant.
	
	\begin{lemma}
	\label{lem4}
			\begin{equation*}
			T(\tilde{L}_0)\subset\tilde{L}_0.
			\end{equation*}
	\end{lemma}
	\begin{proof}
		Since the orthogonal projection $\pi_j:\tilde{W}\longrightarrow W$ is $K$-equivariant for each $j=1,\cdots,q$, $
		\pi_j\circ T:\tilde{W}\longrightarrow W$ is a $K$-equivariant endomorphism.
		Thus $\pi_j\circ T(\tilde{L}_0)\subset W$ is a $K$-representation subspace of $W$.
		It follows from Schur's lemma and Borel-Weil theory that $\pi_j\circ T(\tilde{L}_0)\subset L_0$.
		Concequently $T(\tilde{L}_0)\subset (\tilde{L}_0)$.
	\end{proof}

	We denote by the same notation $T:\tilde{L}_0\longrightarrow \tilde{L}_0$ the restriction of $T:\tilde{W}
	\longrightarrow \tilde{W}$ to $\tilde{L}_0$.

	\begin{theorem}
	\label{thm2}
		The endomorphism $T:\tilde{W}\longrightarrow \tilde{W}$ is the identity map.
	\end{theorem}
	\begin{proof}
		Since the bundle isomorphism $\phi:\tilde{L}\longrightarrow f^*Q$ preserves fiber metrics and $T$ is Hermitian, 
		we have
			\begin{equation*}
			(v_1,v_2)_{\tilde{L}_0}=([e,v_1],[e,v_2])_{\tilde{L}}=([e,Tv_1],[e,Tv_2])_{\tilde{L}}=(Tv_1,Tv_2)_{\tilde{L}
			_0}=(T^2v_1,v_2)_{\tilde{L}_0},
			\end{equation*}
		for any $v_1, v_2\in \tilde{L}_0$.
		Therefore $T^2:\tilde{L}_0\longrightarrow \tilde{L}_0$ is the identity map.
		Since $W$ is $G$-irreducible and $T$ is $G$-equivariant, $T^2:\tilde{W}\longrightarrow \tilde{W}$ is the 
		identity map and so is $T$ because $T$ is semi-positive Hermitian.
	\end{proof}
	
	Consequently, a holomorphic strongly projectively flat $G$-equivariant map $f:M\longrightarrow Gr_p(\C^n)$ is the 
	standard map induced by a $q$-orthogonal direct sum bundle of Hermitian line bundle $L\rightarrow M$.

\end{document}